\documentclass[11pt,reqno]{amsart}
\usepackage{graphicx,amscd}
\usepackage{amsfonts,amsmath,amssymb}
\newcommand{\rl}{{\mathbb{R}}}
\newcommand{\cx}{{\mathbb{C}}}

\newcommand{\coker}{\rm coker\,}

\newcommand{\e}{\varepsilon}

\newcommand{\B}{\mathbb{B}}
\newcommand{\D}{\mathbb D}

\newtheorem{theorem}{Theorem}[section]
\newtheorem{lemma}[theorem]{Lemma}
\newtheorem{prop}[theorem]{Proposition}
\newtheorem{cor}[theorem]{Corollary}

\theoremstyle{definition}
\newtheorem{defn}[theorem]{Definition}

\begin{document}
\title{Discs in hulls of  real immersions to Stein manifolds}
\author{Rasul Shafikov* and Alexandre Sukhov**}
\begin{abstract}
We obtain results on the existence of complex discs in plurisubharmonically  convex hulls of Lagrangian and totally real immersions to Stein manifolds.
\end{abstract}

\maketitle

\let\thefootnote\relax\footnote{MSC: 32E20, 32E30, 32V40, 53D12.
Key words: Stein manifold, symplectic structure, totally real manifold, Lagrangian manifold,  plurisubharmonically convex hull,  complex disc.
}

* Department of Mathematics, the University of Western Ontario, London, Ontario, N6A 5B7, Canada,
e-mail: shafikov@uwo.ca. The author is partially supported by the Natural Sciences and Engineering 
Research Council of Canada.

**Universit\'e des Sciences et Technologies de Lille, 
U.F.R. de Math\'ematiques, 59655 Villeneuve d'Ascq, Cedex, France,
e-mail: sukhov@math.univ-lille1.fr. The author is partially supported by Labex CEMPI.


\section{Introduction}

Denote by  $J_{st}$  the standard complex structure of $\cx^n$; the value of $n$ will be clear from the context. Let $\D = \{ \zeta \in \cx \vert \,\, \vert \zeta \vert < 1 \}$ be the unit disc in $\cx$ equipped with $J_{st}$   and $(M,J)$ be a (almost) complex manifold with a (almost) complex structure $J$. A $J$-{\it complex (or $J$-holomorphic) disc} in $M$ is a  map $f:\D \to M$ holomorphic with respect to $J_{st}$ and $J$; following tradition we often identify $f$ with its image. When the complex structure $J$ is fixed, we simply say a complex (or holomorphic) disc.

If a disc $f$ is continuous on $\overline\D$,  the restriction $f\vert_{\partial\D}$ is called  {\it the boundary of} $f$. We say that the {\it boundary of $f$ is attached or glued} to a subset $K \subset M$, if $f(\partial \D) \subset K$.
Construction of complex discs with boundaries on a prescribed (compact) subset of $M$ is an old and fundamental problem in complex geometry. It plays a major role in the theory of polynomially, 
holomorphically or plurisubharmonically (p.s.h.) convex hulls, see \cite{Fo,St}.

Seminal paper \cite{Gr} by Gromov reveals a profound connection between the hull problems in complex geometry and  symplectic and contact geometry. One of his  most striking results  states that a smooth compact Lagrangian submanifold $E$ of $\cx^n$ contains a boundary of a nonconstant complex disc. In \cite{Gr2} Gromov suggested that his proof must also work for the case of arbitrary Lagrangian immersions to $\cx^n$. This could be a very natural extension of his result since the existence of a Lagrangian immersion is topologically much less restrictive on $E$ than that of a Lagrangian embedding, see \cite{ CiEl,Fo}. Nevertheless, later on it became clear that some technical difficulties occur. A complex disc with boundary glued to $E$ essentially arises in Gromov's method as a disc-bubble smooth on $\overline\D \setminus \{ 1 \}$. When $E$ is smooth, Gromov's removable boundary singularity theorem allows one to extend the map on the whole boundary $\partial \D$ since in the Lagrangian case the area of a  bubble is bounded. The difficulty is to prove analogous removable singularity theorem for discs attached to immersed manifolds. In the present work we propose an approach inspired by the work of Alexander \cite{Al} who adapted Gromov's method to the case of totally real manifolds.

A {\it nearly smooth} complex disc of class $C^m$ is a bounded complex disc  $f:\D \to M$ which 
extends $C^m$-smoothly  to $\partial \D\setminus \{1\}$. We say that a nearly smooth complex disc $f$ {\it is attached }
to a compact subset $K\subset  M$, if $f(\partial\D\setminus\{1\})\subset K$. If additionally $f$ is {\it nonconstant}, we  
call it an A-{\it disc} of class $C^m$, after Herbert~Alexander, who proved the existence of such discs for totally real (not necessarily Lagrangian)  manifolds in $\cx^n$,~\cite{Al}. We simply write A-disc if it is of class $C^\infty$. Alexander's proof combines Gromov's general method with standard complex analytic tools avoiding application of Gromov's compactness theorem. For this reason his approach considerably relies on the affine structure of $\cx^n$. In \cite{ShSu} we extended his result to the case of certain totally real immersions to~$\cx^n$. The goal of the present paper is to extend Alexander's result and the  results of \cite{ShSu} to the case of totally real immersions to some Stein manifolds (the integrability of complex structure in fact, is not needed for some of our results). Here we use the general approach of Gromov. It turns out that Alexander's version of Gromov's result indeed can be generalized for immersions almost literally following Gromov's method (Theorem \ref{GluingDisc3}). As a consequence we obtain that a totally real immersion of dimension $n$ to a complex  $n$-dimensional  manifold of type $\cx \times X$, where $X$ is Stein, is not plurisubharmonically convex (Corollary \ref{CorHulls}).

In the second part of the paper we consider hulls of Lagrangian immersions into Stein manifolds. Our main observation is that the removal of boundary singularity is connected with ``complex" convexity properties of the singular set 
of the immersed manifold. In \cite{ShSu} we used the polynomial convexity working in $\cx^n$; the notion of p.s.h. convexity is suitable in the Stein case. We prove the boundary removable singularity property (and hence, existence of a nonconstant complex disc with boundary glued to $E$) for a  Lagrangian immersion $E$ with isolated locally p.s.h. convex singularities (Theorem \ref{continuity10}). This condition always holds for transverse double intersections (Proposition \ref{InterTheo}). 

\medskip

\noindent{\it Acknowledgment.} We thank S. Ivashkovich for helpful discussions.

\section{Gluing discs to totally real embeddings  and immersions}

The study of symplectic properties of Stein manifolds started in the foundational work of 
Eliashberg-Gromov~\cite{ElGr} and was continued by many authors. Recall that a (almost) complex manifold is called a {\it  Stein manifold} if it admits a smooth strictly plurisubharmonic exhaustion function.
Let $(X,J_X)$ be a Stein manifold of complex dimension $n-1$ with a complex structure $J_X$. Fix a symplectic form $\omega_X$ taming $J_X$ on $X$, i.e., $\omega(v,J_Xv) > 0$ for every nonzero tangent vector $v$, see \cite{CiEl,MS,Si}. We use the notation $(X,\omega_X,J_X)$ for a complex manifold equipped with    a taming symplectic form and a complex structure. 
Denote by $\omega_{st} = (i/2)\sum_{j=1}^n dz_j \wedge d\overline{z}_j$ the standard symplectic form on $\cx^n$, the value of $n$ will be clear from the context. The product $M = \cx \times X$  is also Stein with the complex structure $J = J_{st} \otimes J_X$ and the taming symplectic form 
$\omega = \omega_{st} \oplus \omega_X$. We call such $M$ an {\it admissible} Stein manifold. This class of Stein manifolds  is our main object of study.

It is well known that the following two conditions hold for every Stein manifold:

\begin{itemize}
\item[(A1)] $M$ is complete with respect to the metric $h$ defined by $h(u,v) = (1/2)(\omega(u,Jv) + \omega(v,Ju))$;
\item[(A2)] the structure $J$ is uniformly continuous on $M$ with respect to the metric $h$.
\end{itemize}
These conditions are sufficient in order to apply Gromov's compactness theorem which we will need in 
the proof, see, for example, \cite{Si}. In what follows we everywhere use the norms and distances on 
$M$ induced by $h$.

For a $J$-complex disc $f: \D \to (M,\omega,J)$, $f:\D \ni \zeta = \xi + i\eta \mapsto f(\zeta)$ its  (symplectic) area  is defined by 
\begin{eqnarray}
\label{area}
area(f) = \int_\D f^*\omega.
\end{eqnarray}
The expression 
\begin{eqnarray}
\label{energy}
E(f):=  \frac{1}{2}\int_\D \left ( \left \vert \left \vert \frac{\partial f}{\partial
  \xi} \right\vert \right \vert^2_h + \left\vert\left\vert \frac{\partial f}{\partial
  \eta} \right\vert\right\vert^2_h \right )d\xi \wedge d\eta,
  \end{eqnarray} 
where the norm $\parallel \bullet \parallel_h$ is taken with respect to $h$, is called {\it the energy} of $f$. 
It coincides with the area defined by the metric $h$:
\begin{eqnarray}
\label{EnId}
E(f) = area(f).
\end{eqnarray}
This fundamental  equality is called the energy identity, see, for instance, \cite{MS}.
Similar notions still make sense for holomorphic maps $f:(\Omega,J_{st}) \to (M,\omega,J)$, where $\Omega$ is a domain in $\cx$. Of course, in this case the unit disc $\D$ must be replaced by $\Omega$ in (\ref{area}),(\ref{energy}),(\ref{EnId}).

Recall that a submanifold $E$ in $(M,\omega,J)$ is called {\it totally real} if $T_pE \cap J T_pE = \{ 0 \}$ for every point $p \in E$, and is called {\it Lagrangian} if $\omega\vert_E = 0$ and $\dim E = n$. It is well-known that every Lagrangian manifold is totally real if $J$ is tamed by $\omega$; the converse, of course, in general is not true. In the present paper we consider only totally real submanifolds of maximal possible dimension $n$ in $(M,\omega,J)$.

We begin with the embedded case.   Our first result is the following theorem.

\begin{theorem}\label{GluingDisc2}
Let $E$ be  a compact totally real $C^\infty$-smooth submanifold 
of real dimension $n$ in  an admissible Stein manifold $M$ of complex dimension $n$. Then there exists  
an A-disc attached to~$E$.
\end{theorem}

As we mentioned in Introduction, Alexander proved this result for $M = \cx^n$. We present the proof which does not use the integrability of the complex structure $J$ (hence, goes through in the case where $J$ is an almost complex structure). 
In fact, even if $J$ is integrable, Gromov's method requires the use of almost complex structures.

\begin{proof}[Proof of Theorem~\ref{GluingDisc2}.] First we define suitable manifolds of discs.

  Fix a point $p  \in E$  and  fix also a  non-integer $r > 1$.
  Consider the set of maps 
\begin{equation}\label{mnfd-U}
{\mathcal F} = \left\{ f\in C^{r+1}(\D, M): 
f(\partial \D) \subset E,\ f(1) = p \right\}.
\end{equation}

Denote by  $F$  an open subset of ${\mathcal F}$ which consists of $f$ homotopic to a constant map
 $f^0 \equiv p$ in ${\mathcal F}$. It is well-known that $F$ is a $C^\infty$-smooth complex Banach manifold.
A disc $f$ is holomorphic if and only if it satisfies the Cauchy-Riemann equations
\begin{eqnarray}
\label{CR1}
J \circ df = df \circ J_{st} .
\end{eqnarray}
Let $z = (z_1,...,z_n) \in \cx^n$ be local coordinates on $M$ (not necessarily holomorphic with respect to $J$) in a neighbourhood $U$ of  a point 
$p \in M$. That is, $z:U \to \cx^n$ is a smooth local diffeomorphism and $z(p) = 0$. The direct image $z_*(J) = dz \circ J \circ dz^{-1}$
of $J$ can be viewed as a complex structure in a neighbourhood of the origin. In these coordinates a disc $f$ in its turn can be viewed as a map 
$z:\D \to \cx^n$, $z:\zeta  \mapsto z(\zeta)$, and the Cauchy-Riemann equations can be written in the form convenient for usual analytic tools:

\begin{eqnarray}
\label{CR2}
z_{\bar\zeta} - A(z){\bar z}_{\bar\zeta} = 0 ,
\end{eqnarray}
where a complex $n \times n$ matrix function $A = A(z)$ satisfies the condition $\parallel A \parallel < 1$; here the matrix norm is induced by the Euclidean inner product. More precisely, $A$ is uniquely determined by $J$ as the matrix representation of the operator $(J_{st} + z_*(J))^{-1}(J_{st} - z_*(J))$, which is complex anti-linear with respect to $J_{st}$. In particular, $A(0) = 0$ if $z_*(J)$ coincides with $J_{st}$ at the origin. The integrability of $J$ means that local coordinates can be chosen to be holomorphic, which is equivalent to $A$ vanishing identically in a neighbourhood of the origin, see more details in \cite{NiWo}.

Denote by $V$ the bundle $\overline\D \times TM$ over $\overline \D \times M$. For every disc $f$, consider $V_f = f^*TM$ ,
the pull-back by $f$ of the tangent bundle $TM$. It can be viewed as the restriction of $V$ to the graph of $f$ in $\overline\D \times M$. Denote by $\Omega^{0,1}\D$ the bundle of $(0,1)$-forms on $\D$. Extend this bundle on 
$\overline\D \times M$ keeping the same notation $\Omega^{0,1}\D$. Then we obtain the bundle $\Omega^{0,1}\D \otimes V$ over $\overline\D \times M$.

We introduce the $\overline\partial_J$ operator by setting
\begin{eqnarray}
\label{dbar}
\overline\partial_J f = \frac{1}{2}(df + J \circ df \circ J_{st}) .
\end{eqnarray}
This is just the complex anti-linear part of $df$ with respect to $J$. This operator takes its values in the bundle $\Omega^{0,1}\D \otimes V$. More precisely, for every $\zeta \in \D$ the expression $\overline\partial_Jf(\zeta)$ belongs to  $\Omega^{0,1}_\zeta \D \otimes V_{(\zeta,f(\zeta))}$,  the fibre of $\Omega^{0,1}\D \otimes V$ over $(\zeta,f(\zeta))$.
Conversely, for every continuous section $g = g(\zeta,z) \in \Gamma(\overline\D \times M,\Omega^{0,1}\D \otimes V)$
we may consider the nonhomogeneous Cauchy-Riemann equations
\begin{eqnarray}
\label{CR3}
\overline\partial_J f(\zeta) = g(\zeta,f(\zeta)) .
\end{eqnarray}
See a more detailed discussion in \cite{Gr, IvSh, MS, NiWo}.

An observation of Gromov \cite{Gr} allows us to interpret the nonhomogeneous equation (\ref{CR3}) as the usual Cauchy-Riemann equation (\ref{CR1}) for a suitable almost complex structure determined by $J$ and~$g$. 

Consider the product $\D \times M$ and define there an almost complex structure $J_g$ by
\begin{eqnarray}
\label{StructureLift}
J_g =\left(
\begin{array}{cll}
J_{st} &  0\\
g &  J
\end{array}
\right) .
\end{eqnarray}
Note that $J_g \vert_{TM} = J$.
\begin{lemma}
\label{LiftLemma}
(See [Gr]) A disc $f:\D \to M$ satisfies (\ref{CR3}) if and only if the map $\hat f: \zeta \mapsto (\zeta,f(\zeta))$ is 
$J_g$-complex, i.e.,
satisfies equations (\ref{CR1}) with $J= J_g$. Furthermore, there exists a constant $C_0 = C_0(M,\omega,J)$ such that for every $g \in C^0(\overline\D \times M,\Omega^{0,1}\D \otimes V)$ with $\parallel g \parallel_{L^\infty(\overline\D \times M)} \le C_1 < \infty$ the structure $J_g$ is tamed by the symplectic form $\hat \omega = C_0 C_1\omega_{st} \oplus \omega$.
\end{lemma}

This construction can be easily viewed in local coordinates quite similarly to the equivalence between the coordinate free version of the homogeneous Cauchy-Riemann equations (\ref{CR1}) and their coordinate representation (\ref{CR2}).   Indeed, consider the lift $\hat f :\zeta \mapsto (\zeta,f(\zeta))$ 
of $f$ to $\cx^{n+1} = \cx_{w} \times \cx^n_z$. In coordinates on $\cx^n$ a section  $g$ can
be viewed as a ``vector-valued" form, i.e.,  a $(0,1)$-form on $\D$ with values in $\cx^n$. Hence, 
it can be identified with a map $g:\D \to \cx^n$ (we denote it again by $g$). Then in  coordinates the nonhomogeneous $\overline\partial_J$-equation (\ref{CR3}) can be written the form 
\begin{eqnarray}
\label{NonHomdbar}
z_{\bar\zeta}(\zeta)  - A(z(\zeta)) \overline{z(\zeta)}_{\bar \zeta} = g(\zeta).
\end{eqnarray}
This is equivalent to the following PDEs system for  the lift $\hat f$:

\begin{eqnarray}
\label{LiftCR}
\left\{ \begin{array}{cccc}
& & w_{\overline\zeta} = 0,\\
& &z_{\bar\zeta} - g(w){\bar w}_{\bar\zeta} - A(z) \overline{z}_{\bar \zeta} = 0 .
\end{array} \right.
\end{eqnarray}

This is precisely  the (homogeneous) Cauchy-Riemann equations (\ref{CR2}) for the almost complex structure $J_g$ on $\D \times \cx^n$. 

Denote by $G$ the complex Banach space of all sections  $g \in C^r(\overline\D \times M,\Omega^{0,1}\D \otimes V)$. Set
\begin{equation}\label{mnfd-H}
H = \{ (f,g) \in F \times G: \overline\partial_J f = g \}.
\end{equation}
Then $H$ is a connected submanifold of $F \times G$.  

 We need the following 

\begin{lemma}\label{AlEst1} Suppose that  a sequence $(f_k)$ in ${\mathcal F}$ 
converges to a continuous mapping $f:(\D,\partial\D) \to (M,E)$ uniformly on 
$\overline \D$, and $g_k:= \overline\partial_J f_k$ converges in $G$ to $g \in G$. Assume also that the energies  $E( f_k)$ are uniformly bounded.
Then $f \in  C^{r+1}(\D)$ and $(f_k)$ converges to $f$ 
in ${\mathcal F}$ after possibly passing to a subsequence.
\end{lemma}

This is quite a special case of Gromov's compactness theorem \cite{Gr}, see details in 
\cite[Prop. 5.1.2]{Si}. Indeed,  the lifts $(\hat f_k)$ are attached to the manifold $\hat E := \partial\D \times E$ which is totally real with respect to the almost complex structure $J_{g_k}$. By the hypothesis of the lemma, the areas of $(\hat f_k)$ are uniformly bounded. Since this sequence converges uniformly, bubbles cannot occur and the lemma follows by Gromov's compactness. Note that the simplest version of Gromov's compactness theorem is used here. The proof is based on standard elliptic estimates in the interior of the disc and near the boundary where the reflection principle can be used (for the refection principle and elliptic estimates for $J$-complex curves with totally real boundary data, see, for example, \cite{IvSu}). Technically all elliptic estimates  follow from the classical regularity properties of the integral transform 
$$
T_jf(\zeta) = \frac{1}{2\pi i}\int_{\D} \frac{f(\tau)d\tau \wedge 
d\overline{\tau}}{(\tau - \zeta)^{j}},\ \ j=1,2.
$$ 
For $j=1$ this is the usual Cauchy transform, for $j=2$ this is its formal derivative ( the Beurling transform) and is defined as a singular integral operator (i.e., in the sense of the Cauchy principal value). In the case of $(\cx^n,J_{st})$, considered by Alexander, we have $A = 0$ in the equations (\ref{CR2}) and the proof becomes particularly transparent (in particular, Beurling's transform is not needed). We point out that all estimates  are purely local and can be performed near each interior point in $\D$ or a point in $\partial \D$, and then globalized using finite  open coverings. Note also that the application of Gromov's compactness theorem requires the conditions (A1), (A2) in order to assure standard metric properties of $M$.

\medskip

  Recall that a linear bounded map $u: L \to L'$ between two Banach spaces is called a {\it Fredholm operator} if $\ker u$ and $\coker u$ have finite dimension; {\it the Fredholm index} $\dim \ker u - \dim \coker u$ is homotopy invariant. A $C^1$-map $\phi:M_1 \to M_2$ between two Banach manifolds is called a {\it Fredholm map} if for every point $q \in M_1$ the tangent map $d\phi_q:T_qM_1 \to T_{\phi(q)}M_2$ is a Fredholm operator; the index of the tangent map is independent of $q$ and is called {\it the index} of $\phi$. A point $q \in M_1$ is called a {\it regular point} if  $d\phi_q$ is surjective. A point $p \in M_2$ is called a {\it regular value} if $\phi^{-1}(p)$ does not contain nonregular points (in particular, $\phi^{-1}(p)$ can be empty).

 Consider the canonical  projection $\pi: H \to G$ given by $\pi(f,g) = g$.
 The following properties of $\pi$ are well-known, \cite{Al, Gr, IvSh}:  
\begin{itemize}
\item[(i)] $\pi$ is a map of class $C^1$ 
between two Banach manifolds;
\item[(ii)]  $\pi$ is  a Fredholm map of index $0$; 
\item[(iii)]  the 
constant map $f^0$  is a regular point for $\pi$.
\end{itemize}

The crucial property of $\pi$ is  that   {\it the map $\pi : H \to G$ is not surjective}, 
\cite{Al,Gr,IvSh}. More precisely, in our case it follows from the argument of \cite[p.104]{IvSh}. Note that 
this argument requires that $M = \cx \times X$; this is the reason why we consider admissible Stein 
manifolds and not arbitrary Stein manifolds.

Arguing by contradiction, suppose that an adapted  A-disc of class $C^{r+1}(\D)$ for $E$ does not exist.  In particular, 
$\pi^{-1}(0) = \{ f^0\}$. Then $0 \in G$ is a regular value of $\pi$. If $\pi$ is proper, then  Gromov's argument based on 
Sard-Smale's theorem implies surjectivity of $\pi$ (see \cite{Al,Gr,IvSh}) which is a contradiction. Thus, it remains to show that 
$\pi:H \to G$ is proper.

\smallskip

All we need, is a well-known description of bubbling; we follow \cite{Si}. Arguing by contradiction, suppose that $\pi$ is not proper. Then there exists a sequence $\{(f_k,g_k)\}\subset H$ 
 such that $g_k\to g$ in $G$ but $f_k$ diverge in $F$. Consider the lifts $\hat f_k(\zeta) = (\zeta,f_k(\zeta))$, $\hat f : \D \to \cx \times M$ as in Lemma \ref{LiftLemma}. Every $\hat f_k$ is holomorphic with respect to the almost complex structure $J_{g_k}$ tamed by the symplectic form $\hat \omega$ as in Lemma \ref{LiftLemma}. We measure norms and distances using the metric $h_k$ defined by $\hat\omega$ and $J_{g_k}$. 
  Set $M_k = \sup_{\overline\D}\parallel d\hat f_k(\zeta) \parallel$. There exists $\lambda_k \in \overline\D$ such that $M_k =\parallel d\hat f_k(\lambda_k) \parallel$.
 
 If the sequence $(M_k)$ is bounded, then by Lemma \ref{AlEst1} the sequence $(f_k)$ converges, so we can assume that $M_k \to +\infty$. 
 
 Case 1. The sequence $(\lambda_k)$ converges to a point in $\D$. Without loss of generality assume that it converges to $0$. Consider the renormalized sequence $F_k(\zeta):= \hat f_k(\lambda_k + \zeta/M_k)$. Then the gradients of this sequence are uniformly bounded on every compact subset of $\cx$ and we can assume that it converges uniformly to some  $J_g$-holomorphic $\hat f:\cx \to \D \times M$. This map is bounded since the sequence $(\hat f_k)$ is. Furthermore, $\hat f = (0,f)$ and the map $f$ is holomorphic with respect to $J_X$ (see the equations (\ref{LiftCR})). Since $(X,J_X)$ is a Stein manifold, it admits a strictly p.s.h. function $u$. Then the composition $ u \circ f$ is a subharmonic function  bounded on $\cx \setminus \{ 0 \}$. Therefore, it extends as a bounded subharmonic function on $\cx$. Hence $u \circ f$ is constant and $f$ is constant. But it is easy to check that $\parallel d\hat F_k(0) \parallel = 1$ (see \cite{Si}, p.184 case (a)). This is a contradiction.
 
 Case 2. The sequence $(\lambda_k)$ converges to a point in $\partial \D$. Let $\delta_k = 1 - \vert \lambda_k \vert$. If $(M_k\delta_k)$ is an unbounded sequence, then, arguing as in \cite[p.184, case (b)]{Si}, we reduce the situation to Case 1. Hence, the only possibility is that the sequence $(M_k\delta_k)$ is bounded. Then the standard renormalization argument  \cite[p.184, case (c)]{Si}, produces a noncompact sequence $(\phi_k)$ of automorphisms of $\D$
 such that $(\phi_k)$ converges uniformly on compact subsets of $\overline\D \setminus \{ 1 \}$ to a constant map and such that  the sequence $(\hat f_k \circ \phi_k)$ has uniformly bounded gradients on every compact subset of $\overline\D \setminus \{ 1 \}$.  One can assume that this sequence converges uniformly on every compact subset of $\overline\D \setminus \{ 1 \}$. By Lemma \ref{AlEst1}  the convergence will be in the $C^m$-norm on every compact subset of $\overline\D \setminus \{ 1 \}$ for every $m$ to a $J_g$-complex disc. The limit map is nonconstant, as shown in \cite[p.184, case (c)]{Si}. Hence, the limit is  an A-disc for 
 $\hat E$ of the form $(const,f)$. Then $f$ is an $A$-disc for $E$. This is  a contradiction, which
proves that $\pi$ is proper and completes the proof of Theorem~\ref{GluingDisc2}.
\end{proof} 
 
Note that in the above argument renormalized sequences of discs have uniformly bounded gradients (hence uniformly bounded areas) only on compact subsets of $\overline\D \setminus \{ 1 \}$. Therefore, in general the whole area of a constructed A-disc can  be infinite. If $E$ is a Lagrangian manifold, then areas of compacts in $\overline{\mathbb D}$ are uniformly bounded, and the constructed A-disc $f$ has a bounded area, so it is just the usual bubble.  By Gromov's removable singularity theorem  $f$ extends to the point $1$ as a map of class $C^\infty(\overline\D)$, and we obtain  Gromov's theorem on the existence of a nonconstant holomorphic disc attached to a Lagrangian submanifold in $\mathbb C^n$.

\bigskip

Consider now the case of totally real immersions. Only minor modifications of the above argument are needed.
Let $E = (\tilde E,{\iota})$ be a pair which consists of a compact smooth manifold $\tilde E$
of dimension $n$ and a $C^\infty$-smooth totally real immersion ${\iota}: \tilde E\to M$. We simply 
say that $E$ is an immersed totally real manifold in $M$ identifying it with the image ${\iota}(\tilde E)$. 
We say that an A-disc $f$ is {\it adapted} for the immersion $E$ if for every point $\zeta \in \partial \D \setminus \{ 1 \}$ there exist an open arc $\gamma\subset \partial \D$ containing $\zeta$ and a smooth map 
$ f_b:\gamma \longrightarrow \tilde E$ satisfying ${\iota} \circ f_b = f\vert_\gamma$. 
In other words, in a neighbourhood of every self-intersection point $p$ of $E$ the values of $f$ belong to a 
smooth component of $E$ through $p$. By the cluster set $C(f,\partial \D)$ of a complex disc 
we mean the set of partial limits of the sequences $f(\zeta_k)$ for all sequences $(\zeta_k)$ in $\D$ converging to $\partial\D$, i.e., such that $dist(\zeta_k,\partial\D) \to 0$.

\begin{theorem}\label{GluingDisc3}
Let $E = (\tilde E,{\iota})$ be an immersed totally real manifold in an admissible Stein manifold $M$. Then 
\begin{enumerate}
\item[(i)] There exists an adapted  A-disc $f \in C(\overline\D \setminus \{ 1 \})$ for $E$.
\item[(ii)] If in addition $E$ is Lagrangian, then $f$ is of  bounded area with the cluster set $C(f,\partial\D)$ 
contained in $E$. Its image $C= f(\D)$ is a holomorphic curve of bounded area with the boundary 
$\partial C := \overline C \setminus C$ contained  in $E$.
\end{enumerate}
\end{theorem}

\begin{proof}[Proof of Theorem~\ref{GluingDisc3}.] 
We first deal with part (i).  Fix a point $p = {\iota}(\tilde p) \in E$ which is not a self-intersection point  and  
fix also a  non-integer $r > 1$. Consider the set of pairs 
\begin{equation}\label{mnfd-URS}
{\mathcal F} = \left\{(f,f_b)\in C^{r+1}(\D, M)\times  C^{r+1}(\partial \D, \tilde E): 
f(\partial \D) \subset E,\ f(1) = p,\ {\iota} \circ f_b = f|_{\partial \D} \right\}.
\end{equation}
In other words, together with a (not necessarily complex) disc $f$ we specify a lift of its boundary  to the source manifold $\tilde E$. For brevity we write $f$ instead of $(f,f_b)$.

Denote by  $F$  an open subset of ${\mathcal F}$ which consists of $f$ homotopic to a constant map
 $f^0 \equiv p$ in ${\mathcal F}$. It is well-known that $F$ is a $C^\infty$-smooth complex Banach manifold. We define now $G$ and $H$ as above. Note that  $H$ is a connected submanifold of $F \times G$.  

 An immediate 
but crucial for our goals observation is   that the proof of Lemma \ref{AlEst1} is {\it purely local}, i.e., all estimates and the convergence are established in a neighbourhood of a given boundary point of a disc. This local character of Lemma~\ref{AlEst1} allows us to pass automatically from an  embedded $E$ to 
a globally immersed $E = (\tilde E,\iota)$ in Lemma~\ref{AlEst1}. Indeed, suppose that $q$ is a self-intersection point of $E$ and 
$f(\zeta_0) = q$ for some $\zeta_0\in\partial\D$. It follows from the uniform convergence of the sequence $(f_k)$ 
and the definition of the set ${\mathcal F}$ that there exists a neighbourhood $U$ of $\zeta_0$ such that 
$f(U \cap \partial\D)$ and after passing to a subsequence, $f_k(U \cap \partial\D)$ belong to the same smooth 
component through $p$ of the immersed manifold $E$. This reduces the situation to the embedded case of Gromov's compactness theorem.

The canonical projection $\pi: H \to G$ has the same properties as in the embedded case \cite{Al,IvSh}. Arguing again  by contradiction, assume that an adapted  A-disc of class $C^{r+1}(\D)$ for $E$ does not exist. As above, in order to get a contradiction, we show that 
$\pi:H \to G$ is proper.

Suppose on the contrary,  that $\pi$ is not proper, and consider  a sequence $\{(f_k,g_k)\}\subset H$ 
 as above and the corresponding $M_k$ and $\lambda_k$.  If the sequence $(M_k)$ is bounded, then by 
 Lemma~\ref{AlEst1} the sequence $(f_k)$ converges, so we may assume that $M_k \to +\infty$. 
 
 Case 1. The sequence $(\lambda_k)$ converges to a point of $\D$. Then we obtain a contradiction as in the previous theorem.

Case 2. The sequence $(\lambda_k)$ converges to a point of $\partial \D$. Again, as in the previous proof, this
case can be handled using a normalization. It provides a noncompact sequence $(\phi_k)$ of automorphisms of $\D$
 such that $(\phi_k)$ converges uniformly on compacts subsets of $\overline\D \setminus \{ 1 \}$ to a constant map and such that the sequence $(\hat f_k \circ \phi_k)$ has uniformly bounded gradients on every compact subset of $\overline\D \setminus \{ 1 \}$. Hence we assume that it converges uniformly there. Recall that we are dealing with adapted discs; locally their boundaries are attached (along every sufficiently small open arc) to a single regular branch of $E$, which is an embedded manifold.
 Since Lemma~\ref{AlEst1} is local, it applies in our situation, which gives the 
convergence  also in the $C^{r+1}$-norm on every compact subset in 
$\overline\D\setminus \{1\}$ ( the intersection $K \cap \partial\D$ can be covered by a finite number of open arcs such that every arc is taken by the maps to a single regular branch of $E$). This is the key observation that makes Alexander's construction valid 
in the immersed case. 
 Since locally $E$ is an embedding and the limit  disc is adapted, it is $C^\infty$ smooth on 
$\overline\D \setminus \{ 1 \}$ by the boundary regularity theorem for complex discs with (embedded) totally real boundary value conditions (see, for example, \cite{IvSu,Si}). Therefore, the limit disc is  an adapted A-disc of the form $(const,f)$ for $\hat E$. Then $f$ is an adapted A-disc for $E$.
This contradiction proves 
that $\pi$ is proper and completes the proof of (i).
 
 \smallskip
 
As for (ii),  we obtain that the constructed in part (i)  A-disc $f$ is of 
bounded area. Since $M$ is Stein, there exists a holomorphic proper embedding  $\psi:M \to \cx^n$, see \cite{CiEl,Fo}. Let $\tilde f = \psi \circ f$. Then \cite[Thm 2]{Al} implies that $\tilde f: \D\setminus \tilde f^{-1}(\psi(E)) \to \cx^n \setminus \psi(E)$ is a proper 
map. This implies (ii) and completes the proof. 
\end{proof}

\medskip

Let $K$ be a compact subset in a complex manifold $M$. Its plurisubharmonically (p.s.h.) convex  hull is defined by
$$
\hat K^{psh}_M =\{ p \in M \vert u(p) \leq \sup_K u \,\, 
\mbox{for all continuous p.s.h. functions}\,\, u:M \to \rl \} .
$$
$K$ is called p.s.h.-convex in $M$ is $\hat K^{psh}_M = K$, see, for example, \cite{CiEl,St}.

\begin{cor}
\label{CorHulls}
Let $E$ be a compact totally real immersion of dimension $n$ in an admissible  Stein manifold $M$ of complex dimension $n$. Then $E$ is not p.s.h. convex.
\end{cor}
This follows from (i) Theorem~\ref{GluingDisc3} because by the maximum principle for subharmonic functions an A-disc is contained in the p.s.h. convex hull of $E$ (of course, $E$ does not contain nonconstant holomorphic curves since it is totally real). 

\medskip

{\bf Remarks and comments. } \begin{itemize}
\item[(1)] Theorem \ref{GluingDisc2} and the part (i) of Theorem~\ref{GluingDisc3} remain true for an almost complex manifold 
$(M,\omega,J)  = \cx \times X$ with a symplectic form $\omega$ taming $J$, if it satisfies the assumptions (A1) and (A2), and, if $(X,J_X)$ admits, for example, a strictly p.s.h. function. Indeed, all proofs go through without modifications. The existence of a strictly p.s.h. function implies that a bounded holomorphic map from $\cx \setminus \{ 0 \}$ to $X$ is constant; we used this property considering Case 1 of the above proof. It would be interesting to extend part (ii) of Theorem~\ref{GluingDisc3} to the almost complex case.
\item[(2)] Corollary \ref{CorHulls} is well-known in the case when $M = \cx^n$ and $E$ is a smooth (or even topological) submanifold,  see \cite{CiEl,St}; for totally real immersions in $\cx^n$ it is obtained in \cite{ShSu}.
\item[(3)] It is  well-known that there exist compact totally real manifolds (for example, some $n$-tori in  
$\cx^n$) which do not contain the whole boundary of a nonconstant complex disc, see Alexander \cite{Al2} and Duval-Gayet \cite{DuGa}. An A-disc for such a manifold necessarily has infinite area.
\item[(4)] Ivashkovich and Shevchishin \cite{IvSh} proved the existence of a complex disc $f$ attached to 
an immersed Lagrangian manifold $E$ under the assumption of weak transversality of $E$; in particular, this assumption  holds for transverse double intersections. Their approach follows the original work of Gromov. They proved a  general version of the compactness theorem (including the reflection principle and the removal of singularities) for $J$-complex curves with boundaries glued to a Lagrangian immersion with weakly transverse self-intersections. Their method also works  for  some symplectic manifolds of the form $\cx \times X$ with   tamed almost complex structures satisfying (A1),(A2). 
\item[(5)] It seems quite possible that our results can be extended to  a wider class of Stein manifolds than the one of  admissible Stein manifolds. On the other hand it is clear that some restrictions on a class of Stein manifolds are necessary. Indeed, let $E = \{ z=(z_1,z_2) \in \cx^2 \vert \,\, \vert z_j \vert = 1 , j=1,2 \}$ be the standard torus in $\cx^2$. The function $\rho(z) = dist(z,E)^2$ (the usual Euclidean distance) is strictly p.s.h. in a neighbourhood of $E$ and $M = \{ z: \rho(z) < \varepsilon \}$ is a Stein manifold for $\varepsilon > 0$ small enough. It follows by the maximum principle that every nearly smooth complex disc in $M$ with boundary attached to $E$, is constant. 
\item[(6)] Duval and Sibony \cite{DuSi}   showed how to use an A-disc in order to construct a positive closed current of bidimension (1,1) and of finite mass with the support contained in the polynomially convex hull of a totally real submanifold of $\cx^n$ (their result also holds for totally real immersions \cite{ShSu}). Using methods of symplectic topology, Viterbo \cite{Vi} proved that a totally real submanifold in an $n$-dimensional Stein manifold admitting an exhaustion strictly p.s.h. function with critical points of Morse index $< n$, contains the boundary of a complex curve.
\item[(7)] Let $E$ be a compact subset of $\cx^n$ and $p$ be a point in the polynomially convex hull of $E$. A number of papers is devoted to the construction of a holomorphic disc $f$
centred at $p$  with (a part of) the boundary contained in a prescribed neighbourhood of  $E$.
The first result of this type is due to Poletsky \cite{Po}. It was extended in several directions by Larusson-Sigurdsson \cite{LaSi}, Rosay \cite{Ro}, Drinovec-Drnovsek and Forstneri\v c \cite{DrFo}, Bertrand-Kuzman \cite{BeKu}, and other authors.
\end{itemize}


\section{Lagrangian immersions to Stein manifolds}

We begin with

\begin{defn} 
A closed subset $S$ of a complex manifold $M$ is called locally p.s.h. convex near a point $p \in X$ if there exists a Stein  neighbourhood $U$ of $p$ such that for every sufficiently small $\e> 0$ the intersection $S \cap \overline{\B(p,\e)}$ is p.s.h.-convex in $U$.
\end{defn} 

Our next result establishes the local p.s.h.-convexity near  transverse self-intersection of Lagrangian immersions.

\begin{prop}\label{InterTheo}
Let  $(M,\omega,J)$ be a Stein manifold of complex dimension $n$. Assume that  $L_1$ and $L_2$ are smooth 
Lagrangian submanifolds intersecting transversely at a point $p$. 
Then the union $(L_1 \cup L_2)$ is locally p.s.h.-convex near $p$. 
\end{prop}

\begin{proof}[Proof of Proposition~\ref{InterTheo}]  
The proof can be reduced to the case of $\cx^n$ considered in \cite{ShSu}.
 In local holomorphic coordinates, we can identify $p$ with the origin, and view $M$ as an open ball $\B(0,\varepsilon)$ equipped with the standard complex structure $J_{st}$, where $\varepsilon > 0$ is small enough. Consider the tangent spaces  $E_j = T_0 L_j$, $j=1, 2$. 
  
 \begin{lemma}
The union $E_1 \cup E_2$ is polynomially convex in $\cx^n$.
\end{lemma}

This result was proved in \cite{ShSu} for the case where $E_j$ are Lagrangian spaces with respect to the standard symplectic structure $\omega_{st}$. The same argument holds in our case of general $\omega$ taming $J_{st}$.

\begin{proof} 
If the union $E_1 \cup E_2$ is not polynomially convex in $\cx^n$, there exists a nonconstant 
holomorphic annulus $f$ with the boundary attached to $E_1 \cup E_2$, see~\cite{ShSu} for details. 
This is just a nonconstant  map $f: \Omega \to \cx^n$, holomorphic on the closed annulus 
$\Omega= \{\zeta \in \cx \vert \,\, r_1 \le \vert \zeta \vert \le r_2 \}$ and such that  
$f(r_j\partial\D) \subset E_j$, $j=1,2$; here $0 < r_1 < r_2$. For every $\delta > 0$, the 
 annulus  $\delta f$ also is glued to $E_1 \cup E_2$. Choosing $\delta$ small enough we can assume that $\delta f$ is contained in $\B(0,\varepsilon)$. Since $J_{st}$ is tamed by $\omega$, the symplectic area of $\delta f$  defined by (\ref{area}) (with $\Omega$ instead of $\D$) is strictly positive. Let a 1-form $\lambda$ be a primitive of $\omega$ in $\B(0,\varepsilon)$. Since $E_j$ are Lagrangian spaces, the restrictions 
 $\lambda\vert_{E_j}$, $j=1, 2$, are exact. Then by Stokes' formula the area of $\delta f$ 
 is independent of $\delta$ and therefore is equal to zero. This is a contradiction.
 \end{proof}
Then, by \cite{ShSu}, for every $\varepsilon > 0$ small enough the set $(L_1 \cup L_2) \cap \overline{\B(0,\varepsilon)}$ is  polynomially convex in $\cx^n$. Hence, there exists a smooth 
nonnegative p.s.h. function $\rho$ on $\cx^n$, strictly p.s.h. on 
$\cx^n \setminus (L_1 \cup L_2 \cap \overline{\B(0,\varepsilon)})$, such that $(L_1 \cup L_2)  \cap \overline{\B(0,\varepsilon)} = \rho^{-1}(0)$, see \cite[Theorem 1.3.8]{St}. Transporting this function 
to a neighbourhood of $p$ in $M$ by a local holomorphic chart, we obtain a function with similar properties near $p$ on $M$. This is equivalent to local p.s.h.-convexity of $L_1 \cup L_2$ (see \cite[Prop. 5.13]{CiEl}).
\end{proof}

Now arguing literally as in \cite{ShSu} we obtain the following result:

\begin{theorem}\label{continuity10}
Suppose that a smooth compact Lagrangian immersion $L$  to an admissible Stein manifold $M$ has a finite number of self-intersection points  and is locally p.s.h. convex near every self-intersection point. Then there exists a nonconstant complex  disc continuous on $\overline{\D}$ with boundary attached to~$L$.
\end{theorem}

Indeed, by \cite[Prop. 5.13]{CiEl}, for every self-intersection point there exists a neighbourhood $U$ and a smooth positive p.s.h. function $\rho$ on $U$, strictly p.s.h. on $U \setminus L$ and such that $L\cap U = \rho^{-1}(0)$. Then similarly to Section 5 of \cite{ShSu}, these functions can be glued to a global p.s.h. function in a neighbourhood of $L$. Together with Theorem \ref{GluingDisc3}(ii) this implies the continuity of a complex disc up to the boundary quite similarly to Section 5 of \cite{ShSu}.

In the case of  Lagrangian embeddings we again recover the result of Gromov~\cite{Gr}. 
In view of Proposition~\ref{InterTheo} we have the following

\begin{cor}
\label{continuity20}
Let  $L$ be a smooth compact Lagrangian immersion to an admissible Stein manifold  $M$ with  a finite number of double transverse self-intersection points. 
Then there exists a nonconstant complex  disc  continuous on $\overline{\D}$ with the boundary attached to $L$.
\end{cor}

This result is also a consequence of Ivashkovich-Shevchishin \cite{IvSh}.
Note that Theorem \ref{continuity10}  works in some cases when the result of Ivashkovich-Shevchishin cannot be applied. Indeed, a Lagrangian immersion can be locally p.s.h. convex but not weakly transversal in the sense of \cite{IvSh}, see examples in \cite{ShSu}. It remains an open question whether Corollary \ref{continuity20} holds without any assumption 
on the set of self-intersection points (as Gromov suggested in \cite{Gr2}).


\end{document}